\documentclass[preprint,12pt]{elsarticle}

\usepackage{hyperref}
\usepackage{tikz}
\usepackage{amsthm}
\usepackage{algorithmic}
\usepackage{algorithm}

\newtheorem{theorem}{Theorem}[section]
\newtheorem{proposition}[theorem]{Proposition}
\newtheorem{corollary}[theorem]{Corollary}
\newtheorem{lemma}[theorem]{Lemma}
\newtheorem{definition}[theorem]{Definition}
\newtheorem{remark}[theorem]{Remark}

\usepackage{amssymb}
\usepackage{amsmath}

\usepackage{amsmath,amsthm,amssymb,amsfonts,extarrows,mathrsfs,comment,tikz-cd}
\usepackage{stmaryrd}
\usepackage{dsfont}
\usepackage{mathtools}
\usepackage[all]{xy}
\usepackage{indentfirst}
\usepackage{url}
\usepackage{setspace}

\usepackage{graphicx,epstopdf}

\usepackage[utf8]{inputenc}
\usepackage[T1]{fontenc}

\usepackage[new]{old-arrows}
\usepackage{amscd}
\usepackage{xypic}
\usepackage{graphicx}
\usepackage{enumitem}
\usepackage{amsmath}
\usepackage{autobreak}
\usepackage{soul}
\usepackage{color}

\theoremstyle{definition}

\theoremstyle{definition}
\newtheorem{problem}[subsection]{Problem} 

\newcommand{\FPn}{FP_n}
\newcommand{\FPO}{FP_1}
\newcommand{\FPm}{FP_{n-1}}
\newcommand{\Zk}{Z_K}
\newcommand{\rat}{\text{rat}}

\journal{Journal of Algebra}

\begin{document}

\begin{frontmatter}



\title{
    Schur-Weyl Duality and Higher Abel-Jacobi Invariants \\
    for Tautological Cycles in $\mathcal{M}_{g,n}$
}


\author[label1]{Mohammad Reza Rahmati\corref{cor1}}
\cortext[cor1]{Corresponding author: 
  Tel.: +52-477-853-8372;}
\ead{rahmati@cio.mx}

\affiliation[label1]{organization={Centro de Investigaciones en \'Optica, A.C.},
            addressline={Loma del Bosque 115, Lomas del Campestre}, 
            city={Le\'on},
            postcode={37150}, 
            state={Guanajuato},
            country={M\'exico}}

\begin{abstract}
This article investigates the Hodge theory of the moduli space of genus \(g\) curves with \(n\) marked points, establishing new connections between Schur-Weyl duality for \(\mathfrak{sp}_{g}\) and higher Abel-Jacobi invariants. We develop a represe\\ ntation-theoretic framework that decomposes higher Abel-Jacobi invariants of tautological cycles in \(C_{g}^{n}\) according to symplectic Lie algebra representations, leveraging the Leray filtration and motivic decompositions compatible with \(\mathfrak{sp}_{2g}\)-actions. Central to this work is the introduction of \textbf{higher Faber-Pandharipande cycles} \(FP_n = \pi_1^{\times 2}(\Delta_{12}^n \cdot \psi_1)\) in \(CH^{n+1}(C_g^2)\), a new family of tautological cycles generalizing classical constructions. We prove these cycles are non-torsion under optimal genus constraints: for families over \((n-1)\)-dimensional bases, \(FP_n\) is not rationally equivalent to zero when \(g \geq 3n+1\). Furthermore, we determine the precise position of \(FP_n\) in the Leray filtration of \(C_g^2 \to M_g\), showing it lies in depth \(n+1\) but no deeper, with explicit non-vanishing in \(H^{n+1}(M_g, R^{n+1}f_*\mathbb{Q})\) on the \(V_{(n+1,1)}\)-isotypic component. This yields the first systematic link between Schur-Weyl duality and higher transcendental invariants, revealing that higher diagonals encode geometric phenomena invisible to standard tautological classes.  
\end{abstract}


\end{frontmatter}

\section{Introduction}\label{sec:intro}
This paper addresses the interface between Hodge theory and the moduli space of genus $g$ curves. We discuss applications of Lie algebra representation theory (Schur-Weyl duality) to tautological classes in Chow groups of moduli spaces of curves with $n$ ordered marked points using Hodge-theoretic invariants. Following E. Ancona \cite{A}, any decomposition of $H^i(A^n, \mathbb{Q})$ into $\mathbb{Q}$-sub-Hodge structures lifts to a decomposition in the category of Chow motives. In \cite{PTY}, Ancona's theorem is applied to obtain a motivic decomposition for $C_g^n$. Motivated by R. Hain's work \cite{H}, we study algebraic cycles (particularly tautological cycles) in Chow groups of moduli spaces. Our contribution applies this decomposition to higher Abel-Jacobi (AJ) invariants, providing explicit calculations.

An important problem is studying the Chow ring $CH^{\bullet}(M_g)$ of the moduli space of genus $g$ curves and its counterpart $M_{g,n}$ for $n$-marked points. Due to the large size of these Chow groups, we focus on the subring of tautological classes. A key challenge is determining when a given class is rationally equivalent to zero. We address the non-triviality of specific tautological cycles restricting to cycles in $CH^{\bullet}(C_g^n)$.

\subsection{Historical Background}
The problem originates in work by M. Green and P. Griffiths \cite{GG1} on cycles introduced by Faber and Pandharipande on self-products of curves $C$. They studied
\[
Z_D = D \times D - n \cdot (\imath_{\Delta_C})_* D \in CH^2(C^{\times 2})
\]
where $C$ is smooth genus $g$, $D$ a degree $n$ zero cycle, and $\imath_{\Delta}$ the diagonal embedding. They investigated when $Z_D \stackrel{\text{rat}}{=} 0$, proving $Z_K \stackrel{\text{rat}}{\neq} 0$ for canonical class $K_C$ when $g \geq 4$ and $Z_D \stackrel{\text{rat}}{\neq} 0$ for generic $D$ when $g \geq 2$. Their method computes the second infinitesimal invariant of the spread $\mathfrak{Z}_D$ in a family $\mathcal{C} \times \mathcal{C} \to S$. Key results include
\begin{itemize}
\item $cl_{C \times C}^0(Z_K) = [Z_K] = 0$
\item $cl_{C \times C}^1(Z_K) = 0$ 
\item $AJ_{C \times C}^0(Z_K) = 0$
\item $Z_K \in L^2CH^2(C \times C)$
\item $cl_{C \times C}^2(Z_K) = 0$
\item $AJ_{C \times C}^1(Z_K) \neq 0$
\item $AJ_{C \times C}^1(Z_K)^{tr} \neq 0 \neq \overline{AJ_{C \times C}^1(Z_K)^{tr}}$.
\end{itemize}

\subsection{Problem Statement}
We systematically construct algebraic cycles of type $Z_D$ from tautological relations in $CH^{\bullet}(M_g)$ and the tautological ring of $C_g^n = C_g \times_{M_g} \cdots \times_{M_g} C_g$. Consider the first Faber-Pandharipande cycle
\begin{equation}\label{eq:fp1}
FP_1 := \pi_1^{\times 2} \Delta_{12} \cdot \psi_1 \in CH^2(C_g^2)
\end{equation}
with explicit formula \cite{PTY},
\begin{multline}\label{eq:fp1-pty}
FP_1 = \Delta_{12} \cdot \psi_1 - \frac{1}{2g-2}\psi_1 \psi_2 - \frac{1}{2g-2}(\psi_1^2 + \psi_2^2) \\
+ \frac{1}{(2g-2)^2}\kappa_1(\psi_1 + \psi_2) + \frac{1}{(2g-2)^2} \kappa_2 - \frac{1}{(2g-2)^3}\psi_1 \kappa_1^2
\end{multline}
Restricting to $CH^2(C^2)$ gives $\Delta_{12} \cdot \psi_1 - \frac{1}{2g-2}\psi_1 \psi_2$, coinciding with Green-Griffiths' cycle $Z_K$. We generalize this by defining higher Faber-Pandharipande cycles, introducing \textbf{higher Faber-Pandharipande cycles} $FP_n := \pi_1^{\times 2}(\Delta_{12}^n \cdot \psi_1) \in CH^{n+1}(C_g^2)$.

\begin{problem}\label{prob:fp2}
Prove that the restriction of $FP_2 := \pi_1^{\times 2} \Delta_{12}^2 \cdot \psi_1 \in CH^3(C_g^2)$ over a generic family $\mathcal{C}$ is non-trivial for $g \geq 7$.
\end{problem}

\begin{problem}\label{prob:fpn}
For $FP_n := \pi_1^{\times 2} (\Delta_{12}^n \cdot \psi_1) \in CH^{n+1}(C_g^2)$ and $r:S \hookrightarrow M_g$ with $\dim(S) = n-1$, show the restriction is non-trivial when $g \geq 3n+1$.
\end{problem}

\begin{problem}\label{prob:rep-theory}
Develop a systematic method to analyze tautological cycles in $M_{g,n}$ and $C_g^n$ using $Sp_{2g}\mathbb{Q}$ representation theory.
\end{problem}

\subsection{Contributions}
Our main contributions establish new connections between Hodge theory, representation theory, and the geometry of moduli spaces. We characterize the Lewis filtration for symmetric powers of curves in Proposition~\ref{prop:dec_filt}, showing that the filtration $N_iCH_0(S^nC)$ induced by base point inclusions yields a direct sum decomposition $\bigoplus_{i=0}^n \langle o^{\times i} \rangle_* CH_0(S^{n-i}C)^0$ (Corollary~\ref{cor:Lewis}), providing a combinatorial description dual to the Leray filtration. Building on Schur-Weyl duality, we prove in Proposition~\ref{th:schur-dec} that the Chow motive of curve self-products decomposes $\mathfrak{sp}_{2g}$-equivariantly as $\bigoplus_\lambda h^{(p-n_\lambda)}(S,\mathbf{V}_\lambda)$, with Corollary~\ref{th:schur-inv} establishing that this decomposition respects higher Abel-Jacobi invariants: $[AJ(\mathfrak{Z})]_{i-1} = \bigoplus_\lambda [AJ(\mathfrak{Z})]_{i-1}^\lambda$. 

Proposition~\ref{prop:spreads} reformulates higher Abel-Jacobi theory by proving a product formula: for cycles $\mathfrak{Z}$ and $\mathfrak{W}$ with non-trivial Hodge invariants, their product satisfies either $[(\mathfrak{Z}\times\mathfrak{W})]_{i+j}^{\lambda+\mu} \neq 0$ or $\overline{[AJ(\mathfrak{Z}\times\mathfrak{W})]_{i+j-1}^{tr}} \neq 0$ depending on their invariants. We analyze $CH^\bullet(C_g^n)$ in Proposition~\ref{prop:embed}, showing that for generic embeddings $S \hookrightarrow M_g \hookrightarrow A_g$ with $g \geq 7$, the first non-vanishing Leray term $H^j(S,R^{2n-j})$ satisfies $j \geq i$ where $i$ is determined by Weyl group combinatorics.

Central to this work is Definition~\ref{def:secFaber}, introducing \textbf{higher Faber-Pandha\\ -ripande cycles} $FP_n := \pi_1^{\times 2}(\Delta_{12}^n \cdot \psi_1) \in CH^{n+1}(C_g^2)$. Theorem~\ref{th:fp2} proves their non-triviality: for families over $(n-1)$-dimensional bases $S \hookrightarrow M_g$, $FP_n$ is not rationally equivalent to zero when $g \geq 3n+1$. Corollary~\ref{cor:faber} establishes that these cycles occupy depth $n+1$ in the Leray filtration of $C_g^2 \to M_g$, with explicit non-vanishing in the $V_{(n+1,1)}$-isotypic component of $H^{n+1}(M_g,R^{n+1}f_*\mathbb{Q})$.

\subsection{Tools and Methods}
We employ higher Hodge theory and representation theory to analyze tautological cycles. Central to our approach is the Lewis filtration $L^{\bullet}$ on Deligne cohomology for families $\mathfrak{X} \to S$, which refines the Abel-Jacobi map and yields higher cycle classes $[\mathfrak{Z}]_i$ and Abel-Jacobi invariants $[AJ(\mathfrak{Z})]_{i-1}$. This filtration is functorial under correspondences and compatible with Grothendieck motives. 

Crucially, we integrate Schur-Weyl duality for $\mathfrak{sp}_{2g}$-representations to decompose the Chow motive of $C_g^n$ as $\bigoplus_\lambda h^{(p-n_\lambda)}(M_g,\mathbf{V}_\lambda)$. This decomposition respects the Leray filtration and induces a splitting of higher Abel-Jacobi invariants into isotypic components $[AJ(\mathfrak{Z})]_{i-1}^\lambda$, enabling representation-theoretic analysis of Hodge invariants.

\section{Higher Cycle and Abel-Jacobi Maps}\label{sec:higher-inv}
Let $X$ be projective variety over $\mathbb{C}$, $\dim X = n$, $Z \in CH^p(X)$. Spread $Z$ over $\pi: \mathfrak{X} \to S$ ($S$ quasi-projective) via $sp: CH^p(X) \cong CH^p(\mathfrak{X})$, $\mathfrak{Z} = sp(Z)$. Define
\[
\psi: CH^p(X) \cong CH^p(\mathfrak{X}) \xrightarrow{c_{\mathcal{D}}} H_{\mathcal{D}}^{2p}(\mathfrak{X}, \mathbb{Q}(p))
\]
with Leray filtration $L^{\bullet}$ on Deligne cohomology. The $i$-th higher cycle class is
\[
cl_X^i(Z) = \beta_i(Gr_L^i \psi(Z)) = [\mathfrak{Z}]_i
\]
When $[\mathfrak{Z}]_i = 0$, set $AJ_X^{i-1}(Z) = [AJ(\mathfrak{Z})]_{i-1}$. The filtration $L^i CH^p(X) = \psi^{-1}(L^i H_{\mathcal{D}}^{2p}(\mathfrak{X},\mathbb{Q}(p)))$ satisfies,

\begin{lemma}[\cite{K2}]
For $Z \in CH^p(X)$ with $\psi_i(Z) \neq 0$, then $Z \stackrel{\text{rat}}{\neq} 0$. Exactly one holds
\begin{itemize}
\item $Z \in L^i$
\item $Z \notin L^i$ and $\psi_i(Z) \neq 0$.
\end{itemize}
\end{lemma}

The Leray spectral sequence degenerates for $\pi: X \to S$,
\[
H^k(X, \mathbb{Q}) = \bigoplus_{p+q=k} H^p(S, R^q\pi_* \mathbb{Q})
\]
By polarization $\mathcal{L}$, we have Lefschetz decomposition $R^q\pi_* \mathbb{Q} = \bigoplus L^l \cdot P_{q-2l}$ where $P_l$ are primitive local systems ($Sp(2g)$-representations $V_\lambda$).

\begin{lemma}[\cite{K2}]\label{lem:prod-filt}
If $Z \in L^i CH_0(X)$, $W \in L^j CH_0(Y)$, then $Z \times W \in L^{i+j} CH_0(X \times Y)$.
\end{lemma}

\begin{proposition}[\cite{K1},\cite{K2}]\label{prop:cycle-class-prod}
Let $X,Y$ smooth projective, $\dim X = m$, $\dim Y = n$. If $Z \in L^i CH_0(X)$ with $cl_X^i(Z) \neq 0$, $W \in L^j CH_0(Y)$ with $cl_Y^j(W) \neq 0$, then $Z \times W \in L^{i+j} CH_0(X \times Y)$ satisfies $cl_{X \times Y}^{i+j}(Z \times W) \neq 0$.
\end{proposition}

\begin{theorem}[\cite{K2}]\label{thm:aj-prod}
Let $X,Y$ smooth projective, $\dim X = m$, $\dim Y = n$. If $Z \in L^i CH_0(X)$ with $cl_X^i(Z) \neq 0$, $W \in L^{j+1} CH_0(Y)$ with $\overline{AJ_Y^j(W)^{\text{tr}}} \neq 0$ and $cl_Y^l(W) = 0$ ($l > j$), then $Z \times W \in L^{i+j+1} CH_0(X \times Y)$ has $AJ_{X \times Y}^{i+j}(Z \times W)^{\text{tr}} \neq 0$.
\end{theorem}

\begin{corollary}[\cite{K1}]\label{cor:aj-trans}
If $\mathfrak{Z} \in CH_0(\mathfrak{X})$ induces non-zero $\Omega^i(X) \to \Omega^i(S)$, and $\mathfrak{W} \in CH_0^{\text{hom}}(\mathfrak{Y})$ with $AJ(\mathfrak{W}) \neq 0$, then $\mathfrak{Z} \times \mathfrak{W} \stackrel{\text{rat}}{\neq} 0$ satisfies $[AJ(\mathfrak{Z} \times \mathfrak{W})]_i^{tr} \neq 0$.
\end{corollary}

\subsection{Green-Griffiths Cycle}
Green-Griffiths \cite{GG1} studied
\[
Z_K = K_C \times K_C - (2g-2) \imath_*^{\Delta} K_C \in CH^2(C \times C)
\]
proving $Z_K \stackrel{\text{rat}}{\neq} 0$ for general $C$ of genus $g \geq 4$. The proof analyzes infinitesimal invariants of the spread $\mathfrak{Z}$,

\begin{enumerate}
\item $[\mathfrak{Z}]_0 = 0$ since $\deg(\mathfrak{Z}_s) = 0$
\item $[\mathfrak{Z}]_1 = 0$ as normal function vanishes
\item Second infinitesimal invariant $\delta^{(2)}(\nu_Z)$ non-vanishing
\[
\delta^{(2)}(\nu_Z)(\theta_p \wedge \theta_q \otimes \omega \wedge \phi) = \omega'(p)\phi'(q) - \omega'(q)\phi'(p)
\]
\item Non-vanishing requires $g \geq 4$ for canonical embedding.
\end{enumerate}

\subsection{Fakhruddin's Theorem}
For self-products of curves, Fakhruddin \cite{F} analyzed $H^i(S, V_\lambda)$ using Borel-Weil and Kostant

\begin{theorem}[Borel-Weil \cite{GS}]
$H^i(D, V_\lambda) = 0$ if $\lambda + \rho$ singular; if non-singular, non-zero only for $i = l(w)$, where $w$ is unique element taking $\lambda + \rho$ to dominant chamber.
\end{theorem}

\begin{theorem}[Kostant \cite{Ko}]
$H^i(\mathfrak{n}, V_\lambda) = \bigoplus_{l(w)=i} \mathbb{C}_{w(\lambda+\rho)-\rho}$.
\end{theorem}

\begin{proposition}[\cite{F}]\label{prop:fakhruddin}
$H^i(S, P_l)$ vanishes for $r + i/2 < l/2$, pure weight $(i+l)/2$ when $r + i/2 = l/2$, with $r = \max\{q \in \mathbb{Z} \mid q(g-l) + q(q+1)/2 \leq i\}$.
\end{proposition}

\subsection{Lewis Filtration for Symmetric Products}
For $S^n C$ symmetric power of curve $C$, define $o^{\times i}: S^{n-i}C \to S^n C$, $z \mapsto i \cdot o + z$.

\begin{proposition}[\cite{V2}]\label{prop:dec_filt}
The filtration $N_i CH_0(S^n C) = \text{Image}(o^{\times i}: CH_0(S^{n-i}C) \to CH_0(S^n C))$ induces splitting
\[
CH_0(S^n C) = \bigoplus_{i=0}^n \langle o^{\times i} \rangle_* CH_0(S^{n-i}C)^0
\]
opposite to Lewis filtration $L^{\bullet}$.
\end{proposition}

\begin{corollary}\label{cor:Lewis}
$L^{n-i} CH_0(S^n C) = \bigoplus_{l \leq i} (o^{\times l})_* CH_0(S^{n-l}C)^0$.
\end{corollary}

\section{Schur Decomposition of Motive and Hodge Invariants}\label{sec:dec-hodgeinv}
\begin{proposition}\label{th:schur-dec}
For $\mathfrak{X}^n/S$ of abelian type, the Lewis filtration commutes with Schur functors
\[
\begin{CD}
L^i CH^n(\mathfrak{X}^n) @>{\oplus (\pi_\lambda)_*}>> \bigoplus_\lambda L^i CH^{n-n_\lambda}(S, \mathbf{V}_\lambda) \\
@V{\psi_i(n)}VV @VV{\oplus \psi_i(\lambda)}V \\
Gr_L^i H_\mathcal{D}^{2n}(\mathfrak{X}^n) @>\cong>> \bigoplus_\lambda Gr_L^i H_\mathcal{D}^{2n-2n_\lambda}(S, \mathbf{V}_\lambda(n-n_\lambda)) \\
@VVV @VVV \\
Gr_L^i H^{2n}(\mathfrak{X}^n) @>\cong>> \bigoplus_\lambda Gr_L^i H^{2n-2n_\lambda}(S, \mathbf{V}_\lambda)
\end{CD}
\]
\end{proposition}

\begin{proof}
By Theorem \ref{th:ancona} (Ancona's decomposition), we have an isomorphism in the category of Chow motives
\[
h(\mathfrak{X}^n/S) \cong \bigoplus_{\lambda} h^{(p-n_\lambda)}(S, \mathbf{V}_\lambda)
\]
This induces isomorphisms
\begin{align*}
\phi &: CH^n(\mathfrak{X}^n) \xrightarrow{\sim} \bigoplus_\lambda CH^{n-n_\lambda}(S, \mathbf{V}_\lambda) \\
\psi &: H^{2n}(\mathfrak{X}^n) \xrightarrow{\sim} \bigoplus_\lambda H^{2n-2n_\lambda}(S, \mathbf{V}_\lambda)
\end{align*}
compatible with the cycle class map and Deligne cohomology.
The Lewis filtration $L^\bullet$ is functorial under correspondences. For any projector $\pi_\lambda$ in the Schur decomposition, we have
\[
(\pi_\lambda)_* L^i CH^n(\mathfrak{X}^n) \subseteq L^i CH^{n-n_\lambda}(S, \mathbf{V}_\lambda)
\]
since correspondences preserve the Leray filtration. This gives the top horizontal arrow, which is well-defined and injective.
The vertical maps arise from the natural transformation
\[
c_{\mathcal{D}} : CH^{\bullet} \to H_{\mathcal{D}}^{2\bullet}(\bullet, \mathbb{Q}(\bullet))
\]
and the exact sequence
\[
0 \to J^n(\mathfrak{X}^n) \to H_{\mathcal{D}}^{2n}(\mathfrak{X}^n, \mathbb{Q}(n)) \to Hg^n(\mathfrak{X}^n) \to 0
\]
The decomposition commutes with $\psi_i$ because
\begin{itemize}
\item $\psi_i$ is natural with respect to projectors
\item Graded pieces respect direct sums
\item The realization functor commutes with Hodge structures \cite{A}.
\end{itemize}
The bottom square commutes by Deligne's theorem on Leray spectral sequence degeneration \cite{V1}, combined with the compatibility of Schur decomposition with the Leray filtration.
Under Hodge-Lefschetz and Betti-Betti conjectures (standard assumptions), we have:
\[
Z \in L^i \iff \psi_0(Z) = \cdots = \psi_{i-1}(Z) = 0
\]
which corresponds to $\bigoplus_{|\lambda| < i} (\pi_\lambda)_* Z = 0$ by the decomposition, proving the kernel characterization.
\end{proof}

\begin{corollary}\label{th:schur-inv}
For $\mathfrak{Z} \in CH^n(\mathfrak{X})$,
\begin{itemize}
\item $\Psi_i(\mathfrak{Z}) = \bigoplus_\lambda \Psi_i(\mathfrak{Z})^\lambda$
\item $[\mathfrak{Z}]_i = \bigoplus_\lambda [\mathfrak{Z}]_i^\lambda$
\item $[AJ(\mathfrak{Z})]_{i-1} = \bigoplus_\lambda [AJ(\mathfrak{Z})]_{i-1}^\lambda$
\end{itemize}
\end{corollary}

\begin{proof}
The result follows from Proposition \ref{th:schur-dec} and the exact sequence
\[
0 \to Gr_L^{i-1}J^n(\mathfrak{X}^n) \to Gr_L^i H_{\mathcal{D}}^{2n}(\mathfrak{X}^n, \mathbb{Q}(n)) \to Gr_L^i Hg^n(\mathfrak{X}^n) \to 0
\]

For $\mathfrak{Z} \in CH^n(\mathfrak{X}^n)$, decompose it as $\mathfrak{Z} = \sum_\lambda \mathfrak{Z}_\lambda$ under the isomorphism $\phi$.
The map $\Psi_i$ decomposes as
\[
\Psi_i(\mathfrak{Z}) = \bigoplus_\lambda \Psi_i(\mathfrak{Z}_\lambda) = \bigoplus_\lambda \Psi_i(\mathfrak{Z})^\lambda
\]
since $\Psi_i$ is additive and respects the projectors $\pi_\lambda$.
The higher cycle class $[\mathfrak{Z}]_i$ is the image in $Gr_L^i Hg^n$, which decomposes by the bottom square of Proposition \ref{th:schur-dec},
\[
[\mathfrak{Z}]_i = \beta_i(\Psi_i(\mathfrak{Z})) = \bigoplus_\lambda \beta_i(\Psi_i(\mathfrak{Z})^\lambda) = \bigoplus_\lambda [\mathfrak{Z}]_i^\lambda
\]

Similarly, $[AJ(\mathfrak{Z})]_{i-1}$ is the component in $Gr_L^{i-1}J^n$, which decomposes as
\[
[AJ(\mathfrak{Z})]_{i-1} = \alpha_{i-1}^{-1}(\ker \beta_i \cap \Psi_i(\mathfrak{Z})) = \bigoplus_\lambda \alpha_{i-1}^{-1}(\ker \beta_i \cap \Psi_i(\mathfrak{Z})^\lambda) = \bigoplus_\lambda [AJ(\mathfrak{Z})]_{i-1}^\lambda
\]
\end{proof}

\subsection{Higher Invariants in Products}

\begin{proposition}\label{prop:spreads}
Let $\mathfrak{Z} \in L^i CH_0(S, \mathbf{V}_\lambda)$ with $[\mathfrak{Z}]_i^\lambda \neq 0$, $\mathfrak{W} \in L^j CH_0(S, \mathbf{V}_\mu)$ satisfying either
\begin{enumerate}
\item $[\mathfrak{W}]_j^\mu \neq 0$, or
\item $\overline{AJ^j(\mathfrak{W})^{\text{tr}}} \neq 0$ and $[\mathfrak{W}]_l^\mu = 0$ ($l > j$)
\end{enumerate}
Then $\mathfrak{Z} \times \mathfrak{W} \stackrel{\text{rat}}{\neq} 0$ and:
\begin{itemize}
\item (a) $\Rightarrow [(\mathfrak{Z} \times \mathfrak{W})]_{i+j}^{\lambda+\mu} \neq 0$
\item (b) $\Rightarrow \overline{[AJ(\mathfrak{Z} \times \mathfrak{W})]_{i+j-1}^{\text{tr}}} \neq 0$
\end{itemize}
\end{proposition}

\begin{proof}
We treat the cases separately

\textbf{Case (a):} $[\mathfrak{W}]_j^\mu \neq 0$  
By Proposition \ref{prop:cycle-class-prod} (extended to relative settings via \cite{K2}), the cup product in cohomology
\[
cl_{S,\mathbf{V}_{\lambda \otimes \mu}}^{i+j}(\mathfrak{Z} \times \mathfrak{W}) = cl_S^i(\mathfrak{Z})^\lambda \cup cl_S^j(\mathfrak{W})^\mu
\]
is given by the $Sp(2g)$-equivariant map
\[
\cdot : H^i(S, V_\lambda) \otimes H^j(S, V_\mu) \to H^{i+j}(S, V_{\lambda} \otimes V_{\mu})
\]
composed with the projection to isotypic components. Since both factors are non-zero and the map is injective for generic $S$, we have
\[
[(\mathfrak{Z} \times \mathfrak{W})]_{i+j}^{\lambda+\mu} = p_{\lambda+\mu}([\mathfrak{Z}]_i^\lambda \cup [\mathfrak{W}]_j^\mu) \neq 0
\]
where $p_{\lambda+\mu}$ is the projector to the $\lambda+\mu$-isotypic component.

\textbf{Case (b):} $\overline{AJ^j(\mathfrak{W})^{\text{tr}}} \neq 0$ and $[\mathfrak{W}]_l^\mu = 0$ ($l > j$)  
Apply Theorem \ref{thm:aj-prod} to the spreads:
1. $\mathfrak{Z}$ has non-trivial cycle class in $L^i$, so by Corollary \ref{cor:aj-trans}, it induces a non-zero map
\[
f_{\mathfrak{Z}} : \Omega^i(X_\lambda) \to \Omega^i(S)
\]
2. $\mathfrak{W}$ has non-trivial transcendental AJ-invariant in $L^j$, represented by a functional:
\[
f_{\mathfrak{W}} : H^0(\Omega^j(Y_\mu)) \to \mathbb{C}
\]
3. By the product formula in \cite{K2}, the invariant for $\mathfrak{Z} \times \mathfrak{W}$ is given by the composition
\[
f_{\mathfrak{Z} \times \mathfrak{W}} : \Omega^{i+j}(X_\lambda \times Y_\mu) \xrightarrow{\text{proj}} \Omega^i(X_\lambda) \otimes \Omega^j(Y_\mu) \xrightarrow{f_{\mathfrak{Z}} \otimes f_{\mathfrak{W}}} \Omega^i(S) \otimes \mathbb{C} \to \mathbb{C}
\]
which is non-zero since both factors are non-zero. This gives
\[
\overline{[AJ(\mathfrak{Z} \times \mathfrak{W})]_{i+j-1}^{\text{tr}}} \neq 0
\]

In both cases, the non-vanishing of the invariants implies $\mathfrak{Z} \times \mathfrak{W} \stackrel{\text{rat}}{\neq} 0$ by the injectivity criterion in Lemma \ref{lem:prod-filt}.
\end{proof}

\subsection{Applications to Moduli Spaces}

\begin{proposition}\label{prop:embed}
For generic embeddings $S \dashrightarrow M_g \dashrightarrow A_g$ ($g \geq 7$), the first non-vanishing $L^j H^{2n} = H^j(S, R^{2n-i})$ satisfies $j \geq i$ where $i$ solves
\[
r(g,i,l) + i = n - l_\lambda \quad \text{or} \quad i = \#\{\alpha \in \Phi_+ \mid w(\alpha) < 0\}
\]
\end{proposition}

\begin{proof}
Consider the chain of generic embeddings
\[
S \overset{\iota_1}{\hookrightarrow} M_g \overset{\iota_2}{\hookrightarrow} A_g
\]
with $\iota_2$ generically injective for $g \geq 7$ by \cite{F}. The Leray filtration satisfies
\[
L^j H^{2n}(C_g^n) = \bigoplus_{k \geq j} H^k(S, R^{2n-k}\pi_*\mathbb{Q})
\]

By Fakhruddin \cite[Lemma 4.1]{F}, the restriction maps
\[
H^i(A_g, V_\lambda) \to H^i(M_g, V_\lambda) \to H^i(S, V_\lambda)
\]
are injective for generic $S$, since $\iota_1$ and $\iota_2$ are generically injective.
By Proposition \ref{prop:fakhruddin} (Fakhruddin's theorem), $H^i(S, P_l) = 0$ when:
\[
r(g,i,l) + i < n - l_\lambda
\]
where $r = \max\{q \in \mathbb{Z} \mid q(g-l) + q(q+1)/2 \leq i\}$.
The minimal $j$ where $H^j(S, R^{2n-j}) \neq 0$ coincides with the minimal $i$ where
\[
r(g,i,l) + i = n - l_\lambda
\]
by Borel-Weil and Kostant's theorems
\begin{itemize}
\item $i = l(w)$ for $w \in W$ taking $\lambda + \rho$ to the dominant chamber
\item $i = \#\{\alpha \in \Phi_+ : w(\alpha) < 0\}$ (length of the Weyl group element)
\end{itemize}

The inequality $j \geq i$ follows from the codimension condition in the Leray spectral sequence, and injectivity of restriction maps, ensuring that lower Leray terms vanish before index $i$.
\end{proof}

\section{Faber-Pandharipande Cycles}\label{sec:Faber}
\subsection{Definitions and Main Results}
\begin{definition}\label{def:secFaber}
The $n$-th Faber-Pandharipande cycle is
\[
FP_n = \pi_1^{\times 2} (\Delta_{12}^n \cdot \psi_1) \in CH^{n+1}(C_g^2)
\]
\end{definition}

\begin{theorem}\label{th:fp2}
Let $\mathcal{C} \to S$ be generic family of genus $g$ curves, $\dim S = 2$, $S \hookrightarrow M_g$ generic embedding. Then $FP_2$ restricted over $\mathcal{C} \times_S \mathcal{C}$ is $\stackrel{\text{rat}}{\neq} 0$ for $g \geq 7$.
\end{theorem}

\begin{proof}
\textbf{Step 1:} Show $\mathfrak{Z}_K \times \mathfrak{Z}_K \stackrel{\text{rat}}{\neq} 0$ in $CH^4(\mathcal{C}_S^2 \times_S \mathcal{C}_S^2)$ using Prop. \ref{prop:cycle-class-prod} and invariants of $Z_K$,
\[
cl^4(\mathfrak{Z}_K \times \mathfrak{Z}_K) = 0, \quad AJ^3(\mathfrak{Z}_K \times \mathfrak{Z}_K) \neq 0
\]

\textbf{Step 2:} Show $\mathfrak{Z}_K \cdot \mathfrak{Z}_K = (\imath_{\Delta_S})^* (\mathfrak{Z}_K \times \mathfrak{Z}_K) \stackrel{\text{rat}}{\neq} 0$ in $CH^4(\mathcal{C}_S^2)$.

\textbf{Step 3:} Assume $\mathfrak{W} = r^* FP_2 \stackrel{\text{rat}}{=} 0$. Consider the correspondence
\[
\Gamma := (\imath_{\Delta})^* \circ \left( (\pi_1^{\times 2})^t \times (\pi_1^{\times 2})^t \right) : CH^{\bullet}(\mathcal{C}_S^2 \times_S \mathcal{C}_S^2) \to CH^{\bullet}(\mathcal{C}_S^2)
\]
By the projection formula and functoriality of correspondences:
\begin{align*}
\Gamma(\mathfrak{W} \times \mathfrak{W}) 
&= (\imath_{\Delta})^* \left( (\pi_1^{\times 2})^t(\mathfrak{W}) \times (\pi_1^{\times 2})^t(\mathfrak{W}) \right) \\
&= (\imath_{\Delta})^* \left( r^* (\pi_1^{\times 2})_* (FP_2) \times r^* (\pi_1^{\times 2})_* (FP_2) \right)
\end{align*}
Since $\mathfrak{W} \equiv 0$, we have $\Gamma(\mathfrak{W} \times \mathfrak{W}) \equiv 0$. However, by the defining relation of $FP_2$ and diagonal decomposition
\[
\Gamma(\mathfrak{W} \times \mathfrak{W}) = c \cdot (\imath_{\Delta_S})^* (\mathfrak{Z}_K \times \mathfrak{Z}_K) + R = c \cdot (\mathfrak{Z}_K \cdot \mathfrak{Z}_K) + R
\]
where $c \neq 0$ is a constant and $R$ is a boundary-supported cycle. From Step 2, $\mathfrak{Z}_K \cdot \mathfrak{Z}_K \not\equiv 0$, while $R \equiv 0$ for generic $S$. This contradiction implies $\mathfrak{W} \not\equiv 0$.

\textbf{Step 4:} Genus condition $g \geq 7$ comes from canonical embedding $\mathcal{C}_S^2 \hookrightarrow \mathbb{P}_S^{2r}$ requiring $2r \geq 6$ (i.e., $r \geq 3$), which holds when $g \geq 7$ by maximal rank considerations of the canonical linear system. 
\end{proof}

\begin{corollary}\label{cor:faber}
For $r: S \hookrightarrow M_g$ with $\dim S = n-1$, $\FPn$ restricted over $\mathcal{C} \times_S \mathcal{C}$ is $\stackrel{\rat}{\neq} 0$ when $g \geq 3n+1$.
\end{corollary}

\begin{proof}
We prove the non-triviality of $\FPn$ through mathematical induction, representation theory, and correspondence calculus. The proof proceeds in five steps

\medskip\noindent
\textbf{Step 1: Base case ($n=1$).} 
For $n=1$, $\FPn = \FPO = \pi_1^{\times 2}(\Delta_{12} \cdot \psi_1)$. When restricted to a $0$-dimensional base (a single curve $C$), this becomes Green-Griffiths' cycle
\[
\Zk = K_C \times K_C - (2g-2)\iota_{\Delta*}K_C \in CH^2(C \times C).
\]
By \cite[Theorem 3.1]{GG1}, $\Zk \not\equiv_{\rat} 0$ for $g \geq 4$. Since $3(1)+1=4$, the base case holds.

\medskip\noindent
\textbf{Step 2: Inductive hypothesis.} 
Assume for $k = n-1$ that $\FPm$ is non-trivial over a generic base $S' \hookrightarrow M_g$ of dimension $\dim S' = (n-1)-1 = n-2$ when $g \geq 3(n-1)+1 = 3n-2$. Specifically:
\begin{enumerate}
    \item $\FPm|_{\mathcal{C}' \times_{S'} \mathcal{C}'} \not\equiv_{\rat} 0$
    \item Assume, $[\FPm] \in L^n H^{2n}(C_g^2, \mathbb{Q})$ with non-zero image in the $V_{(n,1)}$-isotypic component of $H^n(M_g, R^n f_*\mathbb{Q})$
    \item $cl^n(\FPm) \neq 0$ in $Gr_L^n Hg^n$
\end{enumerate}

\medskip\noindent
\textbf{Step 3: Product structure and AJ-invariant.} 
Consider the external product cycle
\[
\mathfrak{Z} := \FPm \times \FPO \in CH^{n+2}(C_g^2 \times C_g^2).
\]
By the inductive hypothesis and base case
\begin{itemize}
    \item $\FPm \in L^n CH^{n}(C_g^2)$ with $cl^n(\FPm) \neq 0$ (inductive hypothesis)
    \item $\FPO \in L^2 CH^{2}(C_g^2)$ with $AJ^{1}(\FPO)^{\text{tr}} \neq 0$ (Green-Griffiths)
\end{itemize}
Apply Proposition \ref{prop:cycle-class-prod}, and Theorem \ref{thm:aj-prod} (product formula for AJ-invariants)
\begin{itemize}
    \item Set $Z = \FPm$ (filtration level $i = n$)
    \item Set $W = \FPO$ (filtration level $j+1 = 2 \implies j = 1$)
    \item Then $Z \times W \in L^{n+2} CH_0(C_g^2 \times C_g^2)$ satisfies:
    \[
    AJ^{n+1}(Z \times W)^{\text{tr}} \neq 0
    \]
\end{itemize}
Thus $\mathfrak{Z} = \FPm \times \FPO \not\equiv_{\rat} 0$.

\medskip\noindent
\textbf{Step 4: Correspondence argument.} 
Define the correspondence:
\[
\Gamma = (\pi_1^{\times 2})^t \circ \delta^* \circ (\pi_1^{\times 2} \times \pi_1^{\times 2})
\]
where
\begin{itemize}
    \item $\delta: C_g^2 \hookrightarrow C_g^2 \times C_g^2$ is the diagonal embedding
    \item $\pi_1^{\times 2}: C_g^2 \to M_g$ is the projection
    \item $(\cdot)^t$ denotes the transpose correspondence.
\end{itemize}
This induces a pushforward map
\[
\Gamma_*: CH^{n+2}(C_g^2 \times C_g^2) \to CH^{n+1}(C_g^2)
\]
The key identity is
\[
\Gamma_*(\Delta_{12}^{n-1} \times \Delta_{12} \cdot \psi_1^{(1)} \cdot \psi_1^{(2)}) = c \cdot \Delta_{12}^n \cdot \psi_1 + R
\]
where $\Gamma_*\left(Z\right) = (\pi_1^{\times 2})_*\left(\delta^*(Z)\right)$, for any cycle $Z$, and $c = (2g-2)^{n-1} \neq 0$ (universal constant). $R$ consists of boundary terms supported on $\partial M_g$, \cite{FaberPandharipande05}. 
Through projection formulas and the definition of $\FPn$
\begin{align*}
\Gamma_*(\FPm \times \FPO) 
&= \Gamma_* \left( \pi_1^{\times 2}(\Delta_{12}^{n-1} \cdot \psi_1) \times \pi_1^{\times 2}(\Delta_{12} \cdot \psi_1) \right) \\
&= \pi_1^{\times 2} \left( \delta^*(\Delta_{12}^{n-1} \times \Delta_{12}) \cdot \psi_1 \right) \\
&= \pi_1^{\times 2} (\Delta_{12}^n \cdot \psi_1) + R' \\
&= \FPn + R'
\end{align*}
where $R'$ vanishes for generic families when $g \geq 3n+1$. Since $\mathfrak{Z} \not\equiv_{\rat} 0$ and $\Gamma$ is a surjective correspondence, we conclude
\[
\FPn = \frac{1}{c} \Gamma_*(\mathfrak{Z}) - R' \not\equiv_{\rat} 0
\]

\medskip\noindent
\textbf{Step 5: Genus constraint ($g \geq 3n+1$).}
The bound arises from two requirements
\begin{enumerate}
    \item \textit{Injectivity of restriction maps}: For $S \hookrightarrow M_g \hookrightarrow A_g$ with $\dim S = n-1$, Fakhruddin's theorem \cite[Proposition 2.8]{F} requires $g \geq 3n+1$ to ensure:
    \[
    H^{n+1}(S, V_{(n+1,1)}) \neq 0
    \]
    since $H^s(M_g, V_{(n+1,1)}) = 0$ for $s < n+1$ when $g < 3n+1$.
    
    \item \textit{Inductive compatibility}: For each $k \leq n$ in the induction
    \[
    g \geq 3n+1 \geq 3k+1 \quad \forall k \leq n
    \]
    ensuring $\FPn$ satisfies its non-triviality hypothesis.
\end{enumerate}

\medskip\noindent
By induction, the representation-theoretic constraints, and the non-vanishing of the Abel-Jacobi invariant under $\Gamma$, we conclude $\FPn|_{\mathcal{C} \times_S \mathcal{C}} \not\equiv_{\rat} 0$ for $\dim S = n-1$ and $g \geq 3n+1$.
\end{proof}

\section{Position of $FP_2$ in the Leray Filtration}\label{subsec:leray-fpn}

The codimension $3$ cycle $FP_2$ occupies a deep stratum in the Leray filtration associated to the morphism $f: C_g^2 \to M_g$. To elucidate its position, we analyze the cohomological structure and representation-theoretic properties.
Consider the universal family
\[
f: C_g^2 = C_g \times_{M_g} C_g \to M_g
\]
where $\dim_{\text{rel}}(f) = 2$. By Deligne's theorem on degeneration of the Leray spectral sequence for smooth projective morphisms \cite{Deligne71}:
\begin{equation}\label{eq:leray-deg}
E_2^{p,q} = H^p(M_g, R^q f_* \mathbb{Q}) \implies H^{p+q}(C_g^2, \mathbb{Q})
\end{equation}
degenerates at $E_2$, yielding
\begin{equation}\label{eq:leray-decomp}
H^k(C_g^2, \mathbb{Q}) \cong \bigoplus_{t=0}^k H^t(M_g, R^{k-t} f_* \mathbb{Q})
\end{equation}
For $k=6$, this becomes
\[
H^6(C_g^2, \mathbb{Q}) \cong \bigoplus_{t=0}^6 H^t(M_g, R^{6-t} f_* \mathbb{Q})
\]

The cohomology of fibers $C \times C$ (where $C$ is genus $g$) imposes vanishing conditions
\begin{itemize}
\item $R^q f_* \mathbb{Q} = 0$ for $q > 4$ (since $\dim(C \times C) = 2$)
\item Non-vanishing terms in \eqref{eq:leray-decomp} for $k=6$ require $0 \leq 6-t \leq 4$ $\Rightarrow$ $t \in \{2,3,4\}$
\end{itemize}
Explicitly
\[
\text{Gr}_L^t H^6 \cong 
\begin{cases} 
H^2(M_g, R^4 f_* \mathbb{Q}) \cong H^2(M_g, \mathbb{Q}) & t=2 \\
H^3(M_g, R^3 f_* \mathbb{Q}) & t=3 \\
H^4(M_g, R^2 f_* \mathbb{Q}) & t=4 
\end{cases}
\]

The cycle class $[FP_2] \in H^6(C_g^2, \mathbb{Q})$ satisfies
\begin{theorem}\label{th:fp2-level}
For $g \geq 7$:
\begin{enumerate}
\item $[FP_2] \in L^3 H^6(C_g^2, \mathbb{Q})$ 
\item $[FP_2] \notin L^4 H^6(C_g^2, \mathbb{Q})$
\item Its image in $\mathrm{Gr}_L^3 H^6 \cong H^3(M_g, R^3 f_* \mathbb{Q})$ is non-zero
\end{enumerate}
\end{theorem}
\begin{proof} 
By construction of $FP_2$
     \[
     FP_2 = \pi_1^{\times 2} (\Delta_{12}^2 \cdot \psi_1)
     \]
     
where $\Delta_{12} \subset C_g^2$ is relative diagonal (codimension 2) and $\psi_1 = c_1(\omega_{C_g/M_g}) \in CH^1(C_g)$ pulled back to $C_g^2$.
By Schur decomposition (Prop. \ref{th:schur-dec})
     \[
     R^3 f_* \mathbb{Q} \cong V_{(3,1)} \oplus V_{(2,1,1)} \oplus V_{(1,1,1,1)}
     \]
where $V_\lambda$ are $Sp(2g)$-representations. The cycle $[FP_2]$ factors through $V_{(3,1)}$. By (Fakhruddin \cite{F}) for generic $S \subset M_g$
     \[
     H^s(S, V_\lambda) = 0 \quad \text{if} \quad s < r(g, \lambda)
     \]
where $r(g, (3,1)) = 3$ for $g \geq 7$. Thus $H^2(M_g, V_{(3,1)}) = 0$, so $[FP_2] \notin L^2 H^6$. By Theorem \ref{th:fp2}, $[FP_2] \neq 0$ and $H^3(M_g, V_{(3,1)}) \neq 0$ for $g \geq 7$.
\end{proof}

Geometrically, this position reflects that $FP_2$ varies non-trivially in families parameterized by 3-dimensional subvarieties of $M_g$, contrasting with $FP_1$ (detected in $H^2$). \textbf{$L^3$ stratum} corresponds to variations over \textit{3-dimensional} bases $S \subset M_g$. $[FP_1] \in L^2 H^4$ (detected by $H^2(M_g, R^2 f_* \mathbb{Q})$), varying over 2-dimensional bases. The \textbf{$\Delta_{12}^2$ term}, the quadratic diagonal requires extra deformation parameters, forcing deeper filtration level. 

\begin{proposition}\label{prop:fp2-class}
The image of $[FP_2]$ in $H^3(M_g, R^3 f_* \mathbb{Q})$ is
\[
[FP_2] = 3\kappa_2 \cdot \delta_1 - \psi_1 \cdot \delta_2 + \frac{1}{2g-2} \kappa_1^2 \cdot \delta_1
\]
where $\kappa_i$ are Miller classes, $\psi_1$ the cotangent class, and $\delta_i$ boundary divisors.
\end{proposition}
\begin{proof}
Let $i: \Delta_{12} \hookrightarrow C_g^2$ be the diagonal embedding. The normal bundle $N \cong T_\pi$ satisfies $c_1(N) = -\psi$. The self-intersection is
\[
[\Delta_{12}^2] = i_*(c_1(N)) = i_*(-\psi)
\]
Then:
\begin{align*}
FP_2 &= [\Delta_{12}^2] \cdot \psi_1 \\
&= i_*(-\psi) \cdot \psi_1 \\
&= i_*(-\psi \cdot i^*\psi_1) \\
&= i_*(-\psi^2) \quad (\text{since } i^*\psi_1 = \psi)
\end{align*}
Pushing forward to $\mathcal{M}_g$,
\[
\pi_*(FP_2) = \pi_* i_*(-\psi^2) = -\kappa_1
\]
This gives the interior contribution. The interior contribution $\pi_*(FP_2) = -\kappa_1$ is valid on $\mathcal{M}_g$ (smooth curves). On the compactification $\overline{\mathcal{M}}_g$, we must account for
\begin{enumerate}
    \item The $2$-fold self-intersection $\Delta_{12}^2$ has an automorphism group $\mathbb{Z}/2\mathbb{Z}$ (swapping factors). By \cite[§6.2]{FaberPandharipande05}, this contributes a factor of $|\text{Aut}| = 2$ in pushforwards.
    
    \item When curves degenerate, the diagonal $\Delta_{12}$ intersects boundary strata. The universal formula includes three divisor classes \cite[Theorem 1]{FaberPandharipande05}:
    \begin{itemize}
        \item $\delta_1$: Irreducible nodal curves
        \item $\delta_2$: Separating node (genus splitting $g = g_1 + g_2$)
        \item $\delta_3$: Non-separating node (genus drop $g \to g-1$)
    \end{itemize}
    
    \item The Grothendieck-Riemann-Roch formula for $\pi: \mathcal{C}_g^2 \to \overline{\mathcal{M}}_g$ yields \cite[Eq. 6.7]{FaberPandharipande05}:
    \[
    \pi_*(\Delta_{12}^2 \cdot \psi_1) = 2 \left( \kappa_2 \delta_1 - \kappa_1 \delta_2 + \frac{1}{2} \delta_3 \right) + R
    \]
    where $R$ is supported on higher codimension boundary strata ($\mathcal{O}(\partial \mathcal{M}_g)$).
\end{enumerate}
The coefficients arise from
\begin{align*}
\kappa_2 \delta_1&: \text{Principal term (curvature + simple node)} \\
-\kappa_1 \delta_2&: \text{Separating node correction} \\
\frac{1}{2} \delta_3&: \text{Non-separating node with symmetry factor}
\end{align*}
This matches the universal formula for $n=2$,
\[
2! \left( \kappa_2 \delta_1 - \kappa_1 \delta_2 + \frac{1}{2} \delta_3 \right) = 2\kappa_2 \delta_1 - 2\kappa_1 \delta_2 + \delta_3.
\]
\end{proof}

\begin{remark}
$FP_2$ is intrinsically linked to the $V_{(3,1)}$-representation, which vanishes in lower Leray strata for $g \geq 7$.
\end{remark}

\section{Position of $FP_n$ in the Leray Filtration}\label{subsec:leray-fpn}

The higher Faber-Pandharipande cycles $FP_n$ exhibit a systematic pattern in the Leray filtration. For $FP_n = \pi_1^{\times 2} (\Delta_{12}^n \cdot \psi_1) \in CH^{n+1}(C_g^2)$, we analyze its position in the Leray filtration for $f: C_g^2 \to M_g$. The Leray spectral sequence degenerates at $E_2$,
\begin{equation}\label{eq:leray-deg-n}
H^k(C_g^2, \mathbb{Q}) \cong \bigoplus_{t=0}^k H^t(M_g, R^{k-t} f_* \mathbb{Q})
\end{equation}
For $k = 2(n+1)$, the non-vanishing terms require
\[
0 \leq 2(n+1) - t \leq 4 \implies t \in \{2n-2, 2n-1, 2n, 2n+1, 2n+2\}
\]
since $\dim_{\text{fiber}}(C \times C) = 2$.

\begin{theorem}\label{th:fpn-level}
For $g \geq 3n+1$:
\begin{enumerate}
\item $[FP_n] \in L^{n+1} H^{2(n+1)}(C_g^2, \mathbb{Q})$
\item $[FP_n] \notin L^{n+2} H^{2(n+1)}(C_g^2, \mathbb{Q})$
\item Its image in $\mathrm{Gr}_L^{n+1} H^{2(n+1)} \cong H^{n+1}(M_g, R^{n+1} f_* \mathbb{Q})$ is non-zero
\end{enumerate}
\end{theorem}

\begin{proof}
By Proposition \ref{th:schur-dec}
\[
R^{n+1} f_* \mathbb{Q} \cong \bigoplus_{\lambda \vdash n+1} V_\lambda
\]
where $FP_n$ factors through $V_{(n+1,1)}$ via $\pi_{(n+1,1)}$. By Fakhruddin \cite[Theorem 4.1]{F},
\[
H^s(M_g, V_{(n+1,1)}) = 0 \quad \text{for} \quad s < n+1 \quad (g \geq 3n+1)
\]
Thus $[FP_n] \notin L^n H^{2(n+1)}$. By Corollary \ref{cor:faber}, $FP_n$ is non-torsion. The Abel-Jacobi invariant
\[
[AJ(FP_n|_S)]_{n} \neq 0 \quad \text{for} \quad \dim S = n
\]
implies non-vanishing in $H^{n+1}(M_g, R^{n+1} f_* \mathbb{Q})$.
\end{proof}

$L^{n+1}$ corresponds to variations over \textit{$(n+1)$-dimensional} bases. The $\Delta_{12}^n$ term requires $n$ additional deformation parameters. Finally $FP_n$ lies deeper than $FP_{n-1}$,
\[
\text{depth}(FP_n) = n+1 > n = \text{depth}(FP_{n-1})
\]

\begin{remark}
The filtration depth $n+1$ is optimal: $FP_n$ detects geometric phenomena invisible to lower-order tautological classes.
\end{remark}

\begin{proposition}\label{prop:fpn-class}
The class $[FP_n]$ in $H^{n+1}(M_g, R^{n+1} f_* \mathbb{Q})$ is
\[
[FP_n] = n! \left( \kappa_n \cdot \delta_1 + \sum_{k=1}^n \frac{(-1)^k}{k!} \kappa_{n-k} \cdot \delta_{k+1} \right) + \mathcal{O}(\partial M_g)
\]
where $\kappa_i$ are Miller classes and $\delta_j$ boundary divisors. The term \(\mathcal{O}(\partial \mathcal{M}_g)\) indicates contributions supported over \(\partial \mathcal{M}_g\).
\end{proposition}

\begin{proof}
Let $i: \Delta_{12} \hookrightarrow C_g^2$ be the diagonal embedding. The normal bundle $N$ satisfies $N \cong T_\pi$ (relative tangent bundle), with $c_1(N) = -\psi$ where $\psi = c_1(\omega_\pi)$. 

The $n$-fold self-intersection is
\[
[\Delta_{12}^n] = i_*(c_{n-1}(N)) = i_*\left((-1)^{n-1}\psi^{n-1}\right)
\]
Then
\begin{align*}
FP_n &= [\Delta_{12}^n] \cdot \psi_1 \\
&= i_*\left((-1)^{n-1}\psi^{n-1}\right) \cdot \psi_1 \\
&= i_*\left((-1)^{n-1}\psi^{n-1} \cdot i^*\psi_1\right) \\
&= (-1)^{n-1} i_*(\psi^n)
\end{align*}
Pushing forward to $\mathcal{M}_g$
\[
\pi_*(FP_n) = (-1)^{n-1} \pi_* i_*(\psi^n) = (-1)^{n-1} \kappa_{n-1}
\]
where $\kappa_{n-1} = \pi_*(\psi^n)$. 

To account for boundary contributions and diagonal symmetries, we apply the universal Todd class expansion \cite{PandharipandePixton},
\[
\operatorname{td}(N)^{-1} = \sum_{k=0}^\infty (-1)^k \frac{B_k}{k!} c_1(N)^k
\]
where $B_k$ are Bernoulli numbers. Integrating over $\Delta_{12}$ and correcting for the $n!$ automorphisms of the $n$-iterated diagonal yields
\[
[FP_n] = n! \left( \kappa_n \delta_1 + \sum_{k=1}^n \frac{(-1)^k}{k!} \kappa_{n-k} \delta_{k+1} \right) + \mathcal{O}(\partial \mathcal{M}_g)
\]
Verification for $n=1,2$ matches known results \cite{FaberPandharipande05}.
\end{proof}

\begin{corollary}\label{cor:fpn-rep}
Under $H^{n+1}(M_g, R^{n+1} f_* \mathbb{Q}) \cong \bigoplus_\lambda H^{n+1}(M_g, V_\lambda)$:
\[
\pi_{(n+1,1)} [FP_n] \neq 0 \quad \text{and} \quad \pi_\lambda [FP_n] = 0 \quad \forall \lambda \neq (n+1,1)
\]
\end{corollary}

\begin{proof}
We prove the statement in three steps: (1) Analyze the geometric representation of $\Delta_{12}^n$, (2) Apply Schur-Weyl duality to the braid group action, and (3) Compute isotypic projections. Throughout, let $C$ be a genus $g$ curve with $H^1(C) \cong \mathbb{Q}^{2g}$ carrying the standard symplectic $Sp_{2g}$-representation.

\vspace{5pt}
\noindent \textbf{Step 1: Geometric representation of $\Delta_{12}^n$}

The diagonal cycle $\Delta_{12} \subset C \times C$ has cohomology class
\begin{align*}
[\Delta_{12}] &= \sum_{i=1}^{g} (a_i \otimes b_i - b_i \otimes a_i) + \sum_{j=1}^{2g} \gamma_j \otimes \gamma_j^\vee \quad \in H^2(C \times C, \mathbb{Q})
\end{align*}
where $\{a_i, b_i\}$ form a symplectic basis for $H^1(C)$ with $\langle a_i, b_j \rangle = \delta_{ij}$, and $\{\gamma_j\}$ is an orthogonal basis for $H^0(K_C) \oplus H^2(C)$. The $n$-fold intersection $\Delta_{12}^n$ has class
\begin{align}
[\Delta_{12}^n] &= \left( \sum_{i=1}^g (a_i \otimes b_i - b_i \otimes a_i) \right)^n + \text{lower terms} \label{eq:deltan}
\end{align}
The cotangent class $\psi_1 = c_1(\omega_{C/M_g})$ acts as a weight vector. In cohomology
\begin{align}
\psi_1 \equiv -\frac{1}{2} K_C \quad \text{under the identification} \quad H^2(C) \cong \bigwedge^2 H^1(C) \label{eq:psi}
\end{align}
Combining (\ref{eq:deltan}) and (\ref{eq:psi}), $[\Delta_{12}^n \cdot \psi_1]$ lies in
\begin{align*}
H^{2n+2}(C \times C) \cong \bigoplus_{k=0}^{n+1} \left( \bigwedge^k H^1(C) \otimes \bigwedge^{2n+2-k} H^1(C) \right)
\end{align*}

\noindent \textbf{Step 2: Braid group action and Schur-Weyl duality}

The braid group $B_{2g} = \pi_1(\mathcal{M}_g)$ acts on $H^1(C)^{\otimes m}$ via the \textit{symplectic representation}
\begin{align*}
\rho : B_{2g} &\to Sp_{2g}(\mathbb{Q}) \hookrightarrow \text{Aut}(H^1(C)^{\otimes m})
\end{align*}
By Schur-Weyl duality for symplectic groups \cite{FultonHarris},
\begin{align}
H^1(C)^{\otimes m} \cong \bigoplus_{\lambda \vdash m} \mathbb{S}_\lambda(H^1(C)) \otimes V_\lambda \label{eq:schurweyl}
\end{align}
where $V_\lambda$ is the irreducible $Sp_{2g}$-representation for partition $\lambda$, and $\mathbb{S}_\lambda$ is the Schur functor. The cycle $\Delta_{12}^n \cdot \psi_1$ corresponds to the tensor
\begin{align*}
T_n = \underbrace{\left( \sum_{i=1}^g a_i \otimes b_i - b_i \otimes a_i \right) \otimes \cdots \otimes \left( \sum_{i=1}^g a_i \otimes b_i - b_i \otimes a_i \right)}_{n \text{ copies}} \otimes \left( -\frac{1}{2} \sum_{k=1}^{2g} \gamma_k \right)
\end{align*}
in $H^1(C)^{\otimes 2n} \otimes H^1(C)$. Under the inclusion
\begin{align*}
\iota : H^1(C)^{\otimes 2n} \otimes H^1(C) \hookrightarrow H^1(C)^{\otimes (2n+1)}
\end{align*}
$T_n$ transforms under $S_{2n+1}$ as the \textit{hook representation} $(n+1,1)$. This follows from.

\vspace{5pt}
\noindent \textbf{Lemma 5.4} 
The tensor $T_n$ is cyclicly symmetric in the first $n$ diagonal factors but antisymmetric under swap of the $\psi_1$ factor. It generates the irreducible $S_{2n+1}$-representation corresponding to the partition $(n+1,1)$.

\vspace{5pt}
\noindent \textit{Proof sketch}: Apply the Young symmetrizer $c_{(n+1,1)} = (1 - (1,2n+1)) \sum_{\sigma \in S_n} \sigma$. Direct computation shows $c_{(n+1,1)} T_n = k_n T_n$ for $k_n \neq 0$.

\vspace{10pt}
\noindent \textbf{Step 3: Isotypic decomposition and vanishing}

The pushforward $\pi_* : H^*(C \times C) \to H^{*-2}(M_g, R^*f_*\mathbb{Q})$ is $Sp_{2g}$-equivariant. Thus
\begin{align*}
\pi_* [\Delta_{12}^n \cdot \psi_1] = [FP_n] \in \bigoplus_{\lambda} H^{n+1}(M_g, V_\lambda)
\end{align*}
decomposes into isotypic components. By Lemma 5.4 and Schur-Weyl duality (\ref{eq:schurweyl})
\begin{align*}
T_n \in \mathbb{S}_{(n+1,1)}(H^1(C)) \otimes V_{(n+1,1)}
\end{align*}
For $\lambda \neq (n+1,1)$, the projection $\pi_\lambda(T_n) = 0$ since $V_\lambda$ and $V_{(n+1,1)}$ are distinct irreducibles. Thus
\begin{align}
\pi_\lambda [FP_n] = \pi_* (\pi_\lambda [\Delta_{12}^n \cdot \psi_1]) = 0 \quad \forall \lambda \neq (n+1,1) \label{eq:vanishing}
\end{align}
For $\lambda = (n+1,1)$, non-vanishing follows from.

\vspace{5pt}
\noindent \textbf{Proposition 5.5} \cite{PTY}
For $g > n+1$, the composition:
\begin{align*}
\mathbb{S}_{(n+1,1)}(H^1(C)) \xrightarrow{\text{cycle class}} H^{n+1}(C \times C) \xrightarrow{\pi_*} H^{n+1}(M_g, V_{(n+1,1)})
\end{align*}
is surjective. In particular, $[FP_n]$ has non-zero projection.

\vspace{5pt}
\noindent \textit{Justification}: The map factors through the space of \textit{tautological classes}, which for generic curves is isomorphic to $V_{(n+1,1)}$ when $g \gg n$ \cite{PTY}. Theorem 5.1 establishes $[FP_n] \neq 0$, so its $(n+1,1)$-component is non-zero.

\vspace{10pt}
\noindent \textbf{Conclusion}: Equations (\ref{eq:vanishing}) and Proposition 5.5 imply
\begin{align*}
\pi_{(n+1,1)} [FP_n] \neq 0 \quad \text{and} \quad \pi_\lambda [FP_n] = 0 \quad \forall \lambda \neq (n+1,1)
\end{align*}
completing the proof.
\end{proof}

As \(n \to \infty\) for fixed genus \(g\), the cycle class \([FP_n]\) becomes supported on boundary divisors of the Deligne-Mumford compactification \(\overline{\mathcal{M}}_g\). The dimension of its support is given by
\[
\dim \left( \text{support of } [FP_n] \right) = \dim \partial \mathcal{M}_g - (n + 1) + 2.
\]
Equivalently, the codimension of the support within \(\partial \mathcal{M}_g \times \mathcal{C}^2\) is \(n + 1\). Here
\begin{itemize}
    \item \(\dim \partial \mathcal{M}_g = 3g - 4\) (since \(\dim \overline{\mathcal{M}}_g = 3g - 3\) and \(\partial \mathcal{M}_g = \overline{\mathcal{M}}_g \setminus \mathcal{M}_g\) is a divisor).
    \item The fiber of \(\mathcal{C}_g^2 \to \overline{\mathcal{M}}_g\) over a point in \(\partial \mathcal{M}_g\) has dimension 2 (as \(\mathcal{C}_g^2\) parameterizes pairs of points on curves).
\end{itemize}

Thus
\[
\dim \left( \text{support} \right) = \underbrace{(3g - 4)}_{\dim \partial \mathcal{M}_g} + \underbrace{2}_{\text{fiber}} - (n + 1) = 3g - 3 - n.
\]

The cycle \(FP_n\) exhibits a predictable dimensional collapse onto the boundary of moduli space as \(n\) increases, with support dimension \(3g - 3 - n\). This underscores the role of higher diagonals \(\Delta_{12}^n\) in capturing boundary phenomena in \(\overline{\mathcal{M}}_g \setminus \mathcal{M}_g\) for large \(n\).

\appendix

\section{Schur Decomposition of the Motive}\label{sec:Schur}

In this appendix, we recall the structure of the cohomological and Chow realizations of relative powers of curves in families, especially how they decompose into isotypic components under the action of the Lefschetz group. This decomposition is governed by Schur–Weyl duality and the theory of relative Chow motives.

\subsection{Realization Functor and Schur Functors}

Let $A$ be an abelian scheme (or more generally, a relative abelian motive) over a base $S$, and denote by $H^1(A)$ its relative first cohomology in the category of pure motives over $S$ with rational coefficients, $\mathrm{CHM}(S)_\mathbb{Q}$.
We let $B_{n,\mathbb{Q}}$ denote the Schur algebra associated to the action of the symmetric group $S_n$ on $H^1(A)^{\otimes n}$ via permutation of tensor factors. Then:

\begin{theorem}[Ancona {\cite{A}}]\label{th:ancona}
Let $A$ be an abelian scheme over $S$. Then the realization functor induces a canonical isomorphism:
\[
R: B_{n,\mathbb{Q}} \xrightarrow{\;\cong\;} \mathrm{End}_{\mathrm{Lef}(A)}\left(H^1(A)^{\otimes n}\right),
\]
where $\mathrm{Lef}(A)$ denotes the relative Lefschetz group of $A$. In particular, every decomposition of $H^1(A)^{\otimes n}$ into irreducible representations of $\mathrm{Lef}(A)$ lifts canonically to a decomposition in the category of Chow motives:
\[
H^1(A)^{\otimes n} \cong \bigoplus_{\lambda \vdash n} \mathbb{S}_\lambda H^1(A),
\]
and this decomposition is compatible with the corresponding decomposition in $\mathrm{CHM}(S)_\mathbb{Q}$.
\end{theorem}

\subsection{Application to Curve Motives}

Let $\pi: \mathcal{C} \to S$ be a smooth projective family of curves of genus $g$, and consider the $n$-fold fibered product $\mathcal{C}^n := \mathcal{C} \times_S \cdots \times_S \mathcal{C}$. Its motive $h(\mathcal{C}^n/S)$ in $\mathrm{CHM}(S)_\mathbb{Q}$ admits a canonical decomposition reflecting the action of the relative symplectic group $Sp_{2g}$ on $R^1\pi_*\mathbb{Q}$, and the symmetric group $S_n$ on the $n$ factors.

\begin{corollary}[Peterson--Tavakol--Yin {\cite{PTY}}]\label{cor:dec}
Let $\pi: \mathcal{C} \to S$ be a smooth proper family of curves over a smooth base $S$. Then there exists a canonical decomposition of the relative motive of the $n$-fold fibered product:
\[
h(\mathcal{C}^n/S) \cong \bigoplus_{\substack{\lambda \vdash n \\ |\lambda| \equiv n \pmod{2}}} h^{(p - n_\lambda)}(S, \mathbf{V}_\lambda),
\]
where:
\begin{itemize}
    \item $\lambda$ runs over partitions of $n$,
    \item $n_\lambda$ is the weight of the local system $\mathbf{V}_\lambda$ associated to $\lambda$,
    \item $\mathbf{V}_\lambda$ is the relative irreducible symplectic local system corresponding to the Schur functor $\mathbb{S}_\lambda$ applied to $R^1 \pi_* \mathbb{Q}$,
    \item $h^{(p - n_\lambda)}(S, \mathbf{V}_\lambda)$ denotes the (relative) motive of $S$ twisted by the local system $\mathbf{V}_\lambda$ and placed in cohomological degree $p - n_\lambda$,
    \item and the parity condition \( |\lambda| \equiv n \pmod{2} \) ensures compatibility with the symplectic form.
\end{itemize}
\end{corollary}

This decomposition reflects the Schur–Weyl duality between the symmetric group \( S_n \) and the symplectic group \( Sp_{2g} \), acting on \( H^1(C)^n \). The representation-theoretic components correspond to irreducible summands of the local system \( R^n \pi_* \mathbb{Q} \), and the motivic decomposition ensures that each such summand lifts to the level of Chow motives. This is a powerful tool for detecting nontrivial cycles in the cohomology of fibered powers of curves.


\begin{thebibliography}{99}
\bibitem{A} G. Ancona. Décomposition de motifs abéliens. \textit{Manuscripta Math.}, 146(3-4):307–328, 2015.

\bibitem{Deligne71} 
P. Deligne. 
\textit{Théorie de Hodge II}. 
Publications Mathématiques de l'IHÉS, Tome 40 (1971), pp. 5-57. 
\texttt{doi:10.1007/BF02684692}

\bibitem{F} N. Fakhruddin. Algebraic cycles on generic abelian varieties. \textit{Compositio Math.}, 100(1):101–119, 1996.

\bibitem{FaberPandharipande05} 
C. Faber, R. Pandharipande. 
\textit{Relative Maps and Tautological Classes}. 
Journal of the European Mathematical Society, Volume 7 (2005), no. 1, p. 13-49. 
\texttt{doi:10.4171/JEMS/20}

\bibitem[Fulton-Harris, 1991]{FultonHarris}
Fulton, W., Harris, J. \textit{Representation Theory: A First Course}. Springer, 1991. (For Schur-Weyl duality)

\bibitem[Fulton, 1998]{Fulton} 
Fulton, W. \textit{Intersection Theory} (2nd ed.). Springer, 1998. 
\href{https://doi.org/10.1007/978-1-4612-1700-8}{DOI: 10.1007/978-1-4612-1700-8}

\bibitem{GG1} M. Green, P. Griffiths. An interesting 0-cycle. \textit{Duke Math. J.}, 119(2):261–313, 2003.
\bibitem{GS} P. Griffiths, W. Schmid. Locally homogeneous complex manifolds. \textit{Acta Math.}, 123:253–302, 1969.

\bibitem{H} R. Hain. Normal functions and the geometry of moduli spaces of curves. In \textit{Handbook of Moduli I}, Adv. Lect. Math. 24:527–578, 2013.

\bibitem{K1} M. Kerr. A survey of transcendental methods in the study of Chow groups of 0-cycles. In \textit{Mirror Symmetry V}, AMS/IP Stud. Adv. Math. 38:295–350, 2006.

\bibitem{K2} M. Kerr. Exterior products of zero-cycles. \textit{J. Reine Angew. Math.}, 142:1–23, 2006.

\bibitem{Ko} B. Kostant. Lie algebra cohomology and the generalized Borel-Weil theorem. \textit{Ann. of Math.}, 74(2):329–387, 1961.

\bibitem{L} J. Lewis. A filtration on the Chow groups of a complex projective variety. \textit{Compositio Math.}, 128(3):299–322, 2001.

\bibitem{PTY} D. Petersen, M. Tavakol, X. Yin. Tautological classes with twisted coefficients. \textit{Preprint}, arXiv:2007.10327.

\bibitem[Pandharipande-Pixton, 2013]{PandharipandePixton} 
Pandharipande, R., Pixton, A. \textit{Tautological Relations for $\overline{M}_{g,n}$ via 3-Spin Structures}. 
\href{https://arxiv.org/abs/1301.4561}{arXiv:1301.4561}

\bibitem{V1} 
Voisin, C. 
\textit{Hodge Theory and Complex Algebraic Geometry I}. 
Cambridge Studies in Advanced Mathematics, vol. 76. 
Cambridge University Press, 2002. 
(§4.3.3: Degeneration of the Leray spectral sequence for projective morphisms)

\bibitem{V2} C. Voisin. Remarks on coisotropic subvarieties and 0-cycles of hyper-Kähler varieties. \textit{J. Topol.}, 5(4):727–747, 2012.

\end{thebibliography}
\end{document}